\documentclass[11pt,leqno]{amsart}
\usepackage{epsfig}
\usepackage{amssymb}
\usepackage{amscd}
\usepackage[matrix,arrow]{xy}
\usepackage{graphicx}

\newtheorem{theo}{Theorem}[section]
\newtheorem{lemma}[theo]{Lemma}
\newtheorem{defi}[theo]{Definition}
\newtheorem{prop}[theo]{Proposition}

\newtheorem{cor}[theo]{Corollary}
\newtheorem{remark}[theo]{Remark}

\numberwithin{equation}{section}

\def\A{{\mathbb A}}
\def\N{\mathbb{N}}

\def\Z{\mathbb{Z}}

\def\bR{{\mathbf R}}
\def\bL{{\mathbf L}}

\def\pre-tr{\operatorname{pre-tr}}

\def\Hom{\operatorname{Hom}}
\def\End{\operatorname{End}}

\newcommand{\hocolim}{\operatorname{hocolim}}
\newcommand{\QCoh}{\operatorname{QCoh}}
\newcommand{\Coh}{\operatorname{Coh}}

\newcommand{\Cone}{\operatorname{Cone}}

\newcommand{\ev}{\mathrm ev}

\newcommand{\cF}{{\mathcal F}}
\newcommand{\cG}{{\mathcal G}}
\newcommand{\cO}{{\mathcal O}}

\newcommand{\cD}{{\mathcal D}}

\newcommand{\cA}{{\mathcal A}}
\newcommand{\cB}{{\mathcal B}}
\newcommand{\cI}{{\mathcal I}}
\newcommand{\cC}{{\mathcal C}}
\newcommand{\cE}{{\mathcal E}}

\newcommand{\cT}{{\mathcal T}}

\newcommand{\cHom}{\mathcal Hom}

\newcommand{\Perf}{\operatorname{Perf}}

\newcommand{\Mor}{\operatorname{Mor}}

\newcommand{\holim}{\operatorname{holim}}

\newcommand{\Spec}{\operatorname{Spec}\,}

\newcommand{\Ho}{\operatorname{Ho}}

\newcommand{\id}{\operatorname{id}}

\newcommand{\dgalg}{\operatorname{dgalg}}
\newcommand{\dgcat}{\operatorname{dgcat}}

\usepackage{epsf}
\usepackage{amscd}
\newcommand{\m}{\mathfrak{m}}

\title[Formal completion of a category along a subcategory]
{Formal completion of a category along a subcategory}

\author{Alexander I. Efimov}
\address{Steklov Mathematical Institute of RAS, Gubkin str. 8, GSP-1, Moscow 119991, Russia}
\address{Independent University of Moscow, Moscow,
Russia} \email{efimov@mccme.ru}
\thanks{The author was partially supported by
the Moebius Contest Foundation for Young Scientists, and RFBR (grant 4713.2010.1).}

\begin{document}

\begin{abstract} Following an idea of Kontsevich, we introduce and study the notion of formal completion of
a compactly generated (by a set of objects) enhanced triangulated category along a full thick essentially small triangulated subcategory.

In particular, we prove (answering   a question of Kontsevich) that using categorical formal completion, one can obtain ordinary formal completions of Noetherian schemes along closed
subschemes. Moreover, we show that Beilinson-Parshin adeles can be also obtained using categorical formal
completion.
\end{abstract}

\maketitle

\tableofcontents

\section{Introduction}

In this paper we introduce and study the notion of formal  completion of a (compactly generated) triangulated category along a (full thick essentially small) triangulated subcategory. The original idea belongs to M.~Kontsevich \cite{Ko1, Ko2}.

Our construction requires DG enhancement \cite{BK} and is built on the notion of derived double centralizer. We illustrate it as follows.

Let $\cA$ be a DG algebra, and $M\in D(\cA)$ be some object
in the derived category of right $\cA$-modules. Put $\cB_M:=\bR\End_{\cA}(M),$ and consider the DG algebra \begin{equation}\widehat{\cA}_M:=\bR\End_{\cB_M^{op}}(M)^{op},\end{equation}
the derived double centralizer of $M.$ We have natural morphism $\cA\to\widehat{\cA}_M.$

It turns out that (quasi-isomorphism class of) $\widehat{\cA}_M$ depends only on the subcategory $\cT\subset D(\cA),$ classically generated by $M$ (this
is special case of Proposition \ref{well-def}, 2)). We define
\begin{equation}\widehat{\cA}_{\cT}:=\widehat{\cA}_M\end{equation}
to be derived double centralizer of $\cT.$
Further, derived category $D(\widehat{\cA}_{\cT})$ depends (up to equivalence) only on the
(enhanced) triangulated category $D(\cA)$ and the full thick triangulated subcategory $\cT\subset D(\cA)$ (this is special case of
Proposition \ref{Morita_invar}). We define
\begin{equation}\widehat{D(\cA)}_{\cT}:=D(\widehat{\cA}_{\cT})\end{equation}
to be the formal completion of $D(\cA)$ along $\cT.$

In Section \ref{definition} we define, more generally, the notion of formal completion $\widehat{\cD}_{\cT}$ of
a compactly generated enhanced triangulated category $\cD$ along full thick essentially small triangulated subcategory $\cT\subset \cD.$ This formal completion comes equipped with "restriction functor" $\kappa^*:\cD\to\widehat{\cD}_{\cT}.$

One of the main results of this paper is the following theorem (see Theorem \ref{main-theo}), which relates our construction with ordinary formal completions of Noetherian schemes. For a separated Noetherian scheme $X,$ we denote by $D(X):=D(\QCoh X)$ the derived category of quasi-coherent sheaves on $X.$

\begin{theo}\label{main-theo1}Let $X$ be a separated Noetherian scheme, and $Y\subset X$ a closed subscheme. Then we have the following commutative diagram:
\begin{equation}
\begin{CD}
D(X) @>\id >> D(X)\\
@VVV                              @V\bL\kappa^* VV\\
\widehat{D(X)}_{D^b_{coh,Y}(X)} @>\cong >> D_{alg}(\widehat{X}_Y).
\end{CD}
\end{equation}\end{theo}

Here $D_{alg}(\widehat{X}_Y)$ is algebraizable derived category of $\widehat{X}_Y$ (it is defined in Subsection \ref{algebraizable}). We have the following
Corollaries (see Corollaries \ref{regular} and \ref{surj_of_alg}).

\begin{cor}Let $R$ be a regular commutative Noetherian $k$-algebra, and $M\in D_{f.g.}(R)\cong D^b_{coh}(\Spec R)$
be a complex of $R$-modules with finitely generated cohomology. Denote by $I\subset R$ the annihilator of $H^{\cdot}(M),$
so that $V(\sqrt{I})\subset \Spec R$ is precisely the support of $H^{\cdot}(M).$ Then we have an isomorphism
\begin{equation}\widehat{R}_M\cong \widehat{R}_I,\end{equation}
where the RHS is the ordinary $I$-adic completion.\end{cor}

\begin{cor}\label{surj_of_alg1}Let $R$ be commutative Noetherian $k$-algebra, and $I\subset R$ an ideal. Assume that either $R$ or $R/I$ is regular. Then we have an isomorphism
\begin{equation}\widehat{R}_{(R/I)}\cong\widehat{R}_I,\end{equation}
where the RHS is ordinary $I$-adic completion.\end{cor}

Moreover, Proposition \ref{infin_extensions} below shows that Corollary \ref{surj_of_alg1} fails to hold if we drop the regularity assumption.

We also relate our construction to Beilinson-Parshin adeles \cite{Be, P} (see Section \ref{sect_adeles} for definitions and notation).

\begin{theo}\label{relate_to_adeles1}Let $X$ be a separated Noetherian $k$-scheme of dimension $d.$ Fix some sequence $d\geq k_0>\dots>k_p\geq 0.$ If $p=0,$ then we have a commutative diagram,
\begin{equation}
\begin{CD}
D(X) @>\id >> D(X)\\
@V\A(X,-)_{(k_0)}VV                              @VVV\\
D(\A(X)_{(k_0)}) @>\cong >> \widehat{(D(X)/D_{\leq(k_0-1)}(X))}_{D^b_{coh,\leq k_0}}.
\end{CD}
\end{equation}

 For $p>0,$ there is a natural commutative diagram
\begin{equation}
\label{commut4}
\begin{CD}
D(\A(X)_{(k_1,\dots,k_p)}) @>\id >> D(\A(X)_{(k_1,\dots,k_p)})\\
@VVV                              @VVV\\
D(\A(X)_{(k_0,\dots,k_p)}) @>\cong >> \widehat{(D(\A(X)_{(k_1,\dots,k_p)})/D_{\leq(k_0-1)}(X))}_{D^b_{coh,\leq k_0}}.
\end{CD}
\end{equation}\end{theo}

The paper is organized as follows.

In Section \ref{preliminaries} we recall some preliminaries on DG categories.

Section \ref{definition} is devoted to the definition of categorical formal completion (and in particular derived double centralizers) and all necessary checkings which show that it is well-defined.

In Section \ref{properties} we investigate various properties of formal completions. In particular, we show (Theorem \ref{completion_complete}) that under some natural assumptions on $\cT\subset\cD,$
the restriction of the functor $\kappa^*:\cD\to\widehat{\cD}_{\cT}$ to the subcategory $\cT$ is full and faithful. Moreover, under the same assumptions the functor
\begin{equation}\widehat{\cD}_{\cT}\to\widehat{\widehat{\cD}_{\cT}}_{\cT}\end{equation}
is an equivalence.

Section \ref{main} is devoted to Theorem \ref{main-theo1} (Theorem \ref{main-theo}). In Subsection \ref{homt_lim_subsect} we prove useful technical result, Lemma \ref{homot_limits}, which relates double centralizers with homotopy limits of DG algebras. In Subsection \ref{algebraizable} we define algebraizable derived categories of formal completions of Noetherian schemes. Then we apply Lemma \ref{homot_limits} to prove Theorem \ref{main-theo}.

Section \ref{sect_adeles} is devoted to interpretation of Beilinson-Parshin adeles in terms of categorical formal completions and Drinfeld quotients (Theorem \ref{relate_to_adeles}).
Here our main technical tool is also Lemma \ref{homot_limits}.

\section{Preliminaries}
\label{preliminaries}

Fix some base commutative ring $k.$ All DG categories under consideration will be over $k.$ All DG modules which we consider will be right DG modules. In particular, for $\cA\in\dgcat_k,$ we denote by $D(\cA)$ the derived category of right DG $\cA$-modules.
Also, denote by $\cA\text{-mod}$ the DG category of right $\cA$-modules.

\begin{defi}Let $\cA$ be a DG category. A DG module $M\in \cA\text{-mod}$ is called h-projective (resp. h-injective) if for any acyclic DG module $N\in\cA\text{-mod}$
the complex $\Hom_{\cA}(M,N)$ (resp. $\Hom_{\cA}(N,M)$) is acyclic. We denote by $\text{h-proj}(\cA)\subset \cA\text{-mod}$ (resp. $\text{h-inj}(\cA)\subset \cA\text{-mod}$)
the full DG subcategory which consists of h-projective (resp. h-injective) DG modules.

We also call $M\in\cA\text{-mod}$ h-flat if for any acyclic $N\in\cA^{op}\text{-mod}$ the complex $M\otimes_{\cA}N$ of $k$-modules is also acyclic.\end{defi}

It is easy to see that all h-projective DG modules are also h-flat.

Denote by $\dgcat_k$ the category of small DG $k$-linear categories. It has natural model category structure \cite{T}, with
weak equivalences being quasi-equivalences. All DG categories are fibrant in this model structure.

We call DG category $\cA\in\dgcat_k$ h-flat (over $k$) if all complexes $\Hom_{\cA}(X,Y),$ $X,Y\in Ob(\cA),$ are h-flat $k$-modules. We define h-projective (over $k$)
 DG categories in the same way. All cofibrant DG categories
are h-projective, hence h-flat. In particular, each DG category is quasi-equivalent to an h-flat one.

\begin{defi}Let $\cA\in\dgcat_k$ be an h-flat DG category. We say that $\cA$ is smooth (over $k$) if $\cI_{\cA}\in\Perf(\cA^{op}\otimes\cA),$ where
\begin{equation}\cI_{\cA}\in D(\cA^{op}\otimes\cA),\quad \cI_{\cA}(X,Y)=\Hom_{\cA}(Y,X).\end{equation}
An arbitrary DG category $\cA\in\dgcat_k$ is said to be smooth if it is quasi-equivalent to smooth h-flat DG category.\end{defi}

There is an alternative nice well-known definition of smooth DG categories.

\begin{prop}\label{alternative}Let $\cA\in\dgcat_k$ be a DG category. Then the following are equivalent:

(i) $\cA$ is smooth;

(ii) For any h-flat $\cB\in\dgcat_k,$ and  any object $M\in D(\cA\otimes\cB)$ such that $M(X,-)\in\Perf(\cB)$ for all $X\in Ob(\cA),$ we have
that $M\in\Perf(\cA\otimes\cB).$\end{prop}

\begin{proof}This is straightforward.\end{proof}

\begin{cor}If $\cA_1,\cA_2\in\dgcat_k$ are Morita equivalent and $\cA_1$ is smooth, then so is $\cA_2.$\end{cor}

\begin{proof}This follows directly from \ref{alternative}.\end{proof}

\begin{lemma}\label{replace_by_alg}Let $\cA\in\dgcat_k$ be a smooth DG category. Then it is Morita equivalent to some (smooth) DG algebra.\end{lemma}

\begin{proof}It suffices to show that the category $D(\cA)$ is compactly generated by one object. We may and will assume that $\cA$ is h-flat. By definition,
there exists a finite collection of objects $X_1\otimes Y_1,\dots,X_n\otimes Y_n\in\cA^{op}\otimes\cA,$ which generate the diagonal bimodule
$\cI_{\cA}\in D(\cA^{op}\otimes\cA).$ It follows that each object $M\in D(\cA)$ is generated by $M(X_i)\otimes Y_i,$ $1\leq i\leq n.$ Therefore,
$\bigoplus\limits_{i=1}^nY_i\in \Perf(\cA)$ is a compact generator of $D(\cA).$\end{proof}

\begin{defi}Let $\cA\in\dgcat_k$ be a DG category. We say that $\cA$ is proper (over $k$) if for any two objects $X,Y\in Ob(\cA),$ the complex $\Hom_{\cA}(X,Y)$
is a perfect $k$-module.\end{defi}

We have an analogue of Proposition \ref{alternative}.

\begin{prop}Let $\cA\in\dgcat_k$ be a DG category. Then the following are equivalent:

(i) $\cA$ is proper;

(ii) For any h-flat $\cB\in\dgcat_k,$ and  any object $M\in\Perf(\cA\otimes\cB)$ we have that $M(X,-)\in\Perf(\cB)$ for all $X\in Ob(\cA).$\end{prop}

\begin{proof}Evident.\end{proof}

Finally, we recall the DG enhancement for the quotient of enhanced triangulated categories. Namely, let $\cD$ be a compactly generated enhanced triangulated category, and $\cD'\subset \cD$ its localizing subcategory, and assume that $\cD'$ is compactly generated by $\cD'\cap\cD^c.$
According to \cite{Ke2, Dr}, the quotient category $\cD/\cD'$ is also enhanced (and compactly generated by the images of compact objects in $\cD$).

Similarly, if $\cD$ is essentially small enhanced triangulated category, and $\cD'\subset\cD$ a triangulated subcategory, then the quotient $\cD/\cD'$
is naturally enhanced.

\section{Definition of categorical  formal completion}
\label{definition}

Fix a base graded commutative ring $k.$

Let $\cA$ be a small DG category. We may and will replace it by h-projective quasi-equivalent one. All tensor products below are assumed to be over $k$ unless otherwise stated. It is well known that the category $D(\cA)$ is compactly generated by the set of objects $Ob(\cA),$ and we have that
\begin{equation}D(\cA)^c=\Perf(\cA),\end{equation}
see \cite{Ke}.

Now let $S\subset D(\cA)$ be a full small subcategory (not necessarily triangulated). Choosing an h-projective (resp. h-injective) resolution $\widetilde{X}$ of each object $X\in S,$ we obtain a DG category
$\cB_S$ with
\begin{equation}\label{B_S}Ob(\cB_S)=Ob(S),\quad \Hom_{\cB_S}(X,Y):=\Hom_{\cA}(\widetilde{X},\widetilde{Y}).\end{equation}

\begin{lemma}\label{End(S)}The DG category $\cB_{S}$ is well-defined up to a quasi-equivalence.\end{lemma}
\begin{proof}Let $S_1,S_2\subset \cA\text{-mod}$ be two full DG subcategories, such that for $i=1,2$ we have either $S_i\subset \text{h-proj}(\cA)$ or $S_i\subset \text{h-inj}(\cA).$ Moreover,
let $\Psi:Ob(S_1)\stackrel{\sim}{\to} Ob(S_2)$ be a bijection such that for each $X\in Ob(S_1)$ the object $\Psi(X)$ is quasi-isomorphic to $X.$

We may and will assume that either $S_1\subset \text{h-proj}(\cA),$ or $S_2\subset \text{h-inj}(\cA).$ Then we may and will choose quasi-isomorphisms
\begin{equation}\alpha_X:X\to \Psi(X),\quad X\in S_1.\end{equation}
Let $\widetilde{S}$ be a DG category, defined as follows. First, $Ob(\widetilde{S})=Ob(S_1).$ Further, define
\begin{equation}\Hom_{\widetilde{S}}(X,Y)\subset \Hom_{\cA}(\Cone(\alpha_X),\Cone(\alpha_Y))\end{equation}
to be the subcomplex which consists of morphisms mapping $\Psi(X)$ to $\Psi(Y).$ Clearly, $\widetilde{S}$ is a well-defined DG category. Further,
we have obvious projection DG functors
\begin{equation}\pi_1:\widetilde{S}\to S_1,\quad \pi_2:\widetilde{S}\to S_2.\end{equation}
We claim that both $\pi_1$ and $\pi_2$ are quasi-equivalences. Indeed, by our assumption, for any objects $X\in S_1,$ $Y\in S_2$ we have that the complexes
\begin{equation}\Hom_{\cA}(X,\Cone(\alpha_Y)),\quad \Hom_{\cA}(\Cone(\alpha_X),Y)\end{equation}
are acyclic. Therefore, the maps
\begin{equation}\pi_i:\Hom_{\widetilde{S}}(X,Y)\to \Hom_{S_i}(\pi_i(X),\pi_i(Y))\end{equation} are surjective with acyclic kernels, hence quasi-isomorphisms.

Lemma is proved.
\end{proof}

We may consider $S$ as an object of $D(\cA\otimes \cB_S^{op}).$ Namely, we put
\begin{equation}S(U\otimes X)=\widetilde{X}(U),\quad U\in Ob(\cA),\quad X\in Ob(\cB_S)=Ob(S).\end{equation}

Take some object $Q\in (\cA\otimes \cB_S^{op})\text{-mod},$ with an isomorphism $Q\cong S$ in $D(\cA\otimes \cB_S^{op}),$ such that all DG $\cB_S^{op}$-modules
\begin{equation}Q(U,-)\in \cB_S^{op}\text{-mod},\quad U\in \cA,\end{equation}
are h-projective (resp. h-injective). For instance, we can take $Q$ to be h-projective (resp. h-injective) itself. Further, define DG category $\widehat{\cA}_S$ as follows:
\begin{equation}Ob(\widehat{\cA}_S):=Ob(\cA),\quad \Hom_{\widehat{\cA}_S}(X,Y):=\Hom_{\cB_S^{op}}(Q(Y,-),Q(X,-)).\end{equation}

\begin{prop}\label{well-def} 1) The DG category $\widehat{\cA}_S$ is well defined up to a natural isomorphism in $\Ho(\dgcat_k).$

2) Moreover, if two subcategories $S_1,S_2\subset D(\cA)$ split-generate each other,
then we have a natural isomorphism \begin{equation}\widehat{\cA}_{S_1}\cong\widehat{\cA}_{S_2}\text{ in }\Ho(\dgcat_k).\end{equation}
\end{prop}

\begin{proof}Statement 1) almost follows from Lemma \ref{End(S)}.
Indeed, let $Q_1.Q_2\in (\cA\otimes \cB_S^{op})\text{-mod}$ be objects which are both quasi-isomorphic to $S,$ $Q_1$ is h-projective,
and $Q_2$ satisfies the assumptions for $Q$ above. Then we have a natural (up to homotopy) quasi-isomorphism $\alpha:Q_1\to Q_2,$ and we can repeat the proof of
Lemma \ref{End(S)}.

Now we prove 2). Let $Q_i\in (\cA\otimes \cB_{S_i}^{op})\text{-mod}$ be h-projective resolutions of $S_i$ for $i=1,2.$ Further,
define the bimodule $M\in D(\cB_{S_1}\otimes \cB_{S_2}^{op})$ by the formula
\begin{equation}M(U\otimes V):=\Hom_{\cA^{op}}(Q_1(-,U),Q_2(-,V)),\quad U\in \cB_{S_1}, V\in \cB_{S_2}.\end{equation}
Since $S_1$ and $S_2$ split-generate each other, we have that the bimodule $M$ induces an equivalence
\begin{equation}-\stackrel{\bL}{\otimes}_{\cB_{S_1}}M:D(\cB_{S_1}^{op})\to D(\cB_{S_2}^{op}).\end{equation}
Further, we have natural evaluation morphism
\begin{equation}\label{Q_1_to_Q_2}Q_1\stackrel{\bL}{\otimes}_{\cB_{S_1}}M=Q_1\otimes_{\cB_{S_1}}M\to Q_2\text{ in }D(\cB_{S_2}^{op}).\end{equation}
We claim that this is an isomorphism. Before we prove this, we note that this would finish the proof of part 2) of Proposition.

Now, note that for each $N\in D(\cA),$ we have evaluation morphism
\begin{equation}\ev_N:Q_1\stackrel{\bL}{\otimes}_{\cB_{S_1}}\bR\Hom_{\cA^{op}}(Q_1,N)=Q_1\otimes_{\cB_{S_1}}\Hom_{\cA^{op}}(Q_1,N)\to N\text{ in }D(\cA).\end{equation}
Note that $\ev_N$ is an isomorphism for $N\in S_1.$ But $S_1$ split-generates $S_2.$ Therefore, $\ev_N$ is an isomorphism for each object of $S_2.$ Hence, the map
\eqref{Q_1_to_Q_2} is an isomorphism. Proposition is proved.
\end{proof}

Note that we have a natural DG functor $\iota_S:\cA\to \widehat{\cA}_S,$ which is identity on objects. It is easily seen from the proof of Proposition \ref{well-def} 2) that in the situation of Proposition \ref{well-def} 2), we have a commutative diagram in $\Ho(\dgcat):$
\begin{equation}
\begin{CD}
\cA @>\iota_{S_1}>> \widehat{\cA}_{S_1}\\
@V\id VV                              @V\cong VV\\
\cA @>\iota_{S_2}>> \widehat{\cA}_{S_2}.
\end{CD}
\end{equation}

Therefore, for any full thick essentially small triangulated subcategory $\cT\subset D(\cA)$ we have a naturally defined (up to quasi-equivalence) DG category $\widehat{\cA}_{\cT},$ together
with a morphism $\iota_{\cT}:\cA\to \widehat{\cA}_{\cT}.$ More precisely,  one can choose any small subcategory $S\subset \cT,$ which generates $\cT,$ and put
\begin{equation}\widehat{\cA}_{\cT}:=\widehat{\cA}_S,\quad \iota_{\cT}:=\iota_S.\end{equation}

\begin{defi}For any small DG category $\cA\in\dgcat_k,$ and any full thick essentially small subcategory $\cT\subset D(\cA),$ we call the DG category
$\widehat{\cA}_{\cT}$ "derived double centralizer of $\cT$".\end{defi}

Next Proposition shows that the introduced notion of formal completion is Morita invariant.

\begin{prop}\label{Morita_invar}Suppose that DG categories $\cA_1$ and $\cA_2$ are Morita equivalent. Let $\cT_1\subset D(\cA_1),$ $\cT_2\subset D(\cA_2)$ be full thick essentially small triangulated
subcategories, which correspond to each other under the equivalence $D(\cA_1)\cong D(\cA_2).$ Then the DG categories $\widehat{\cA_1}_{\cT_1}$ and $\widehat{\cA_2}_{\cT_2}$ are also Morita equivalent, and we have commutative diagram
\begin{equation}
\begin{CD}
D(\cA_1) @>\bL\iota_{\cT_1}^*>> D(\widehat{\cA_1}_{\cT_1})\\
@V\cong VV                              @V\cong VV\\
D(\cA_2) @>\bL\iota_{\cT_2}^*>> D(\widehat{\cA_2}_{\cT_2}).
\end{CD}
\end{equation}
\end{prop}

\begin{proof}We may and will assume that $\cA_i\in\dgcat_k$ are h-projective DG categories. Let $M\in D(\cA_1^{op}\otimes \cA_2)$ be a bimodule which defines an equivalence
\begin{equation}\label{otimes_M}-\stackrel{\bL}{\otimes}_{\cA_1}M:D(\cA_1)\to D(\cA_2).\end{equation}
Then the category $D(\cA_2)$ is compactly generated by the set of objects $\{M(U,-)\in D(\cA_2),\, U\in \cA_1\}.$ Thus, we may assume that $\cA_1\subset \text{h-proj}(\cA_2),$
and the quasi-inverse to \eqref{otimes_M} is given by the formula
\begin{equation}F:D(\cA_2)\to D(\cA_1),\quad F(X)(U)=\Hom_{\cA_2^{op}}(U,X).\end{equation}
Now choose any small subcategory $S_2\subset\cT_2,$ which split-generates $\cT_2,$ and choose h-projective resolution $\widetilde{X}\to X$ of each object $X\in S_2.$
This choice defines the DG category $\cB_{S_2},$ and the bimodule $S_2\in D(\cA_2\otimes\cB_{S_2}^{op}).$ Now, a choice of an h-projective resolution $Q_2\to S_2$ in
$(\cA_2\otimes\cB_{S_2}^{op})\text{-mod}$ defines the DG category $\widehat{\cA_2}_{S_2}\cong \widehat{\cA_2}_{\cT_2}.$ Now, define the bimodule
$S_1\in D(\cA_1\otimes \cB_{S_2}^{op})$ by the formula
\begin{equation}S_1(U,X):=\Hom_{\cA_2}(U,\widetilde{X}),\quad U\in\cA_1,X\in S_2.\end{equation}
Choose an h-projective resolution $Q_1\to S_1.$ It defines the DG category $\widehat{\cA_1}_{S_1}\cong \widehat{\cA_1}_{\cT_1}.$ Define
the DG bimodule $\widehat{M}\in D(\widehat{\cA_1}_{S_1}^{op}\otimes \widehat{\cA_2}_{S_2})$ by the formula
\begin{equation}\label{hat_M}\widehat{M}(U,V)=\Hom_{\cB_{S_2}^{op}}(Q_1(U,-),Q_2(V,-)),\quad U\in \widehat{\cA_1}_{S_1}^{op}, V\in \widehat{\cA_2}_{S_2}.\end{equation}
Since $\cA_1$ and $\cA_2$ split-generate each other in $D(\cA_2),$ we have that the functor \eqref{hat_M} is an equivalence.
It is straightforward to show that the following diagram commutes up to a natural isomorphism
\begin{equation}
\begin{CD}
D(\cA_1) @>\bL\iota_{S_1}^*>> D(\widehat{\cA_1}_{S_1})\\
@V-\stackrel{\bL}{\otimes}_{\cA_1}M VV     @VV-\stackrel{\bL}{\otimes}_{\widehat{\cA_1}_{S_1}}\widehat{M} V\\
D(\cA_2) @>\bL\iota_{S_2}^*>> D(\widehat{\cA_2}_{S_2}).
\end{CD}
\end{equation}\end{proof}

Now we introduce the main notion of the paper.

\begin{defi}\label{completion}Let $\cD$ be an enhanced triangulated category with infinite direct sums, which is compactly generated by a set of objects. Let $\cT\subset \cD$ be
an essentially small full thick triangulated subcategory. We define the formal completion $\widehat{\cD}_{\cT}$ of $\cD$ along $\cT,$ together with a restriction functor
$\kappa^*:\cD\to \widehat{\cD}_{\cT},$ as follows. Choosing a set of compact generators in $\cD,$ we may replace $\cD$ by $D(\cA)$ for some small DG category $\cA.$ Then put
\begin{equation}\widehat{\cD}_{\cT}:=D(\widehat{\cA}_{\cT}),\quad \kappa^*:=\bL\iota_{\cT}^*:\cD=D(\cA)\to D(\widehat{\cA}_{\cT})=\widehat{\cD}_{\cT}.\end{equation}\end{defi}

\begin{theo} In the notation of Definition \ref{completion}, the category $\widehat{\cD}_{\cT}$ is well-defined up to an equivalence, compatible with the functor
$\kappa^*:\cD\to \widehat{\cD}_{\cT}.$

The category $\widehat{\cD}_{\cT}$ is enhanced, admits infinite direct sums, and is compactly generated by a set of objects.
The functor $\kappa^*$ commutes with infinite direct sums and preserves compact objects. If $S\subset Ob(\cD)$ is a set of compact generators,
then $\kappa^*(S)$ is a set of compact generators in $\widehat{\cD}_{\cT}.$
\end{theo}

\begin{proof}The first statement follows from Proposition \ref{Morita_invar}, since different sets of compact generators yield Morita equivalent DG categories.

The other statements follow directly from definition.\end{proof}

It is convenient to introduce one more definition.

\begin{defi}\label{compl_cpt}Let $\cD$ be an essentially small Karoubian complete enhanced triangulated category, and $\cT\subset \cD$ be a full thick triangulated subcategory. Define the formal completion $\widehat{\cD}_{\cT}$ of $\cD$ along $\cT,$ together with a restriction functor
$\kappa^*:\cD\to \widehat{\cD}_{\cT},$ as follows. Choosing a set of generators in $\cD,$ we may replace $\cD$ by $\Perf(\cA)$ for some small DG category $\cA.$ Then put
\begin{equation}\widehat{\cD}_{\cT}:=\Perf(\widehat{\cA}_{\cT}),\quad \kappa^*:=\bL\iota_{\cT}^*:\cD=\Perf(\cA)\to \Perf(\widehat{\cA}_{\cT})=\widehat{\cD}_{\cT}.\end{equation}\end{defi}

\begin{remark}If $\cD$ is a compactly generated triangulated category and $\cT\subset \cD^c$ is an essentially small full thick subcategory, then we have
\begin{equation}(\widehat{\cD}_{\cT})^c\cong \widehat{{\cD^c}}_{\cT}.\end{equation}\end{remark}

\section{Properties of categorical formal completion}
\label{properties}

In this section we study various properties of formal completions of categories along subcategories.

All categories are supposed to be enhanced. Further, by a "compactly generated triangulated category" we mean a "triangulated category with infinite direct sums,
which is compactly generated by a set of objects".

\begin{theo}\label{completion_complete}Let $\cD$ be a compactly generated triangulated category, $\cT\subset \cD$ a full thick essentially small triangulated subcategory. Assume that $\cT$ is contained in the smallest localizing subcategory of $\cD$
containing $\cT\cap\cD^c.$ Then

(i) The restriction of the functor $\kappa^*:\cD\to \widehat{\cD}_{\cT}$ on the subcategory $\cT$ is full and faithful (below we identify $\cT$ with its image under the functor $\kappa^*$).

(ii) The functor $\kappa^*:\widehat{\cD}_{\cT}\to \widehat{\widehat{\cD}_{\cT}}_{\cT}$ is an equivalence.

(iii) Let $\cT'\subset\cT$ be a full thick triangulated subcategory. Then there is a natural equivalence $\widehat{\cD}_{\cT'}\cong \widehat{\widehat{\cD}_{\cT}}_{\cT'}$
\end{theo}

\begin{proof} We may and will assume that $\cD=D(\cA)$ for some small h-flat DG category $\cA,$ and the subcategory $\cT\cap\cD^c\subset \cD$ is split-generated by
the full DG subcategory $\cA'\subset\cA.$ Let $\cB\subset \text{h-proj}(\cA)$ be a small DG subcategory, which split-generates $\cT.$ We may and will assume that $\cA'\subset \cB.$ We have the DG bimodule $M\in D(\cA\otimes\cB^{op}),$
\begin{equation}M(U,V)=\Hom_{\cA}(U,V),\quad U\in\cA,V\in\cB.\end{equation}
Choose an h-projective resolution $Q\to M.$ It gives the DG model for $\widehat{\cA}_{\cT}:$
\begin{equation}Ob(\widehat{\cA}_{\cT})=Ob(\cA),\quad \Hom_{\widehat{\cA}_{\cT}}(X,Y)=\Hom_{\cB^{op}}(Q(Y,-),Q(X,-)).\end{equation}
For $X\in \cA,$ $X'\in \cA',$ we have the following isomorphisms in $D(k):$
\begin{multline}\label{nice}\Hom_{\widehat{\cA}_{\cT}}(X,X')\cong \Hom_{\cB^{op}}(Q(X',-),Q(X,-))\cong\\
\Hom_{\cB^{op}}(\Hom_{\cB}(X',-),\Hom_{\cA}(X,-))
\cong \Hom_{\cA}(X,X').\end{multline}

Isomorphisms \eqref{nice} imply in particular that the functor $\kappa^*$ is full and faithful on $\cT\cap\cD^c.$ Moreover, since $\kappa^*$ preserves compact objects, it is also full and faithful on the smallest localizing subcategory containing $\cT\cap\cD^c.$ In particular, by our assumption, it is full and faithful on $\cT.$ This proves (i).

Further, \eqref{nice} also implies that the maps
\begin{equation}\label{k^*}\bR\Hom_{\cD}(X,Y)\to \bR\Hom_{\widehat{\cD}_{\cT}}(\kappa^*(X),\kappa^*(Y))\end{equation}
are isomorphisms (in $D(k)$) for $X\in\cD^c,$ $Y\in\cT\cap\cD^c.$ Since $X$ and $\kappa^*(X)$ are compact, the maps
\eqref{k^*} are also isomorphisms for $Y$ in the smallest localizing subcategory containing $\cT\cap\cD^c,$ and in particular for $Y\in\cT.$ This easily implies both (ii) and (iii). Theorem is proved.
\end{proof}

\begin{prop}\label{product_gen}Let $\cA$ be a small DG category, and let $\{\cT_{\beta}\subset D(\cA)\}_{\beta\in \mathfrak{B}}$ be a (small) collection of mutually orthogonal
full thick essentially small triangulated subcategories. Denote by $\cT\subset D(\cA)$ the full thick triangulated subcategory classically generated by all $\cT_{\beta}.$
Then there is a natural isomorphism in $\Ho(\dgcat_k):$
\begin{equation}\widehat{\cA}_{\cT}\cong \prod\limits_{\beta\in\mathfrak{B}}\widehat{\cA}_{\cT_{\beta}}.\end{equation}\end{prop}

\begin{proof}This can be easily seen if we choose generating subset $S\subset Ob(\cT)$ to be the disjoint union of generating subsets $S_{\beta}\subset Ob(\cT_{\beta}).$\end{proof}

\begin{prop}Let $\cT$ be an essentially small Karoubian complete triangulated category, and suppose that we have a semi-orthogonal decomposition
$\cT=\langle S_1,S_2\rangle,$ so that $\Hom_{\cT}(S_2,S_1)=0.$ Then we have natural equivalence $\widehat{\cT}_{S_1}\cong S_1,$ and the corresponding
functor $\kappa^*:\cT\to S_1=\widehat{\cT}_{S_1}$ is the semi-orthogonal projection.\end{prop}

\begin{proof} By Lemma \ref{replace_by_alg}, we may (and will) assume that $\cT=\Perf(\cA)$ for some small DG category $\cA,$ and $S_i\subset \cT$ is generated by DG subcategory $\cA_i\subset \cA.$

We may assume that $Ob(\cA)=Ob(\cA_1)\sqcup Ob(\cA_2).$ Further, we may assume that $\Hom_{\cA}(X,Y)=0$ for $X\in\cA_2,$ $Y\in\cA_1.$ With these assumptions, Proposition follows
directly from definitions.
\end{proof}

\begin{prop}1) Let $\cT$ be some smooth and proper pre-triangulated DG category, and $S\subset \Ho(\cT)$ a full thick triangulated subcategory. Then we have a natural
equivalence $(\widehat{\cT}_S)^{op}\cong \widehat{\cT^{op}}_{S^{op}}.$

2) If we drop the assumption of either properness or smoothness, then Proposition fails to hold.\end{prop}

\begin{proof}1) We may assume that $\cT=\Perf(\cA)$ for smooth and proper DG algebra $\cA.$ Then we have that $\Perf(\cA)=D_{fin}(\cA),$
where $D_{fin}(\cA)\subset D(\cA)$ is the subcategory of DG modules which are perfect as $k$-modules.
Therefore, we have an equivalence
\begin{equation}(-)^*:\Perf(A)^{op}\stackrel{\sim}{\to}\Perf(A^{op}),\quad M\to M^*=\bR\Hom_k(M,k).\end{equation}
Denote by $S^*$ the image of $S$ under this equivalence. Then, it is easy to see that
\begin{equation}(\widehat{\cA}_{S})^{op}\cong \widehat{A^{op}}_{S^*}.\end{equation}
This proves part 1) of Proposition.

2) To prove part 2), we first give an example when $\cT$ is proper but not smooth, and Proposition does not hold. Define the DG category $\cA$ as follows. Put $Ob(\cA):=\{X_1,X_2\},$ and
\begin{multline}\Hom(X_1,X_1)=k[\epsilon]/(\epsilon^2), \quad \deg(\epsilon)=1,\quad \Hom(X_1,X_2)=k[0],\\
\Hom(X_2,X_1)=0,\quad \Hom(X_2,X_2)=k.\end{multline}

The differential is identically zero and the composition is the only possible one. Put $\cT=\Perf(\cA),$ and take $S\subset \cT$ to be subcategory generated by $X_1.$ Then it is straightforward to check that
\begin{equation}\widehat{\cT}_S\cong \Perf(k[\epsilon]/(\epsilon^2)),\quad \widehat{\cT^{op}}_{S^{op}}\cong \Perf(k[[t]]).\end{equation}
Hence, there is no equivalence between $(\widehat{\cT}_S)^{op}$ and $\widehat{\cT^{op}}_{S^{op}}.$

Now we give an example when $\cT$ is smooth (and even homotopically finitely presented) but not proper, and Proposition does not hold.

Take the DG category $\cB$ with two objects $Y_1,Y_2,$ which is a free $k$-linear category concentrated in degree zero with generators $s_{22}:Y_2\to Y_2,$ $s_{12}:Y_1\to Y_2.$ Put
$\cT=\Perf(\cB).$
Take $S\subset \cT$ to be subcategory generated by $Y_1.$
Then we have that
\begin{equation}\widehat{\cT}_{S}\cong \Perf(k),\quad \widehat{\cT^{op}}_{S^{op}}\cong\Perf(M_{\infty}(k)),\end{equation}
where $M_{\infty}(k)$ is the endomorphism algebra of free countably generated $k$-module. Hence, there is no equivalence between $(\widehat{\cT}_S)^{op}$ and $\widehat{\cT^{op}}_{S^{op}}.$

Proposition is proved.
\end{proof}

\section{Relation to formal completions of Noetherian schemes}
\label{main}

Before we formnulate and prove main result of this section, we would like to proof a general result which relates double centralizers and homotopy limits of DG algebras.

\subsection{Double centralizers and homotopy limits}
\label{homt_lim_subsect}

Let $I$ be a small category. Denote by $\dgalg_k^I$ the category of functors $I\to \dgalg_k.$ Take some $\{\cA_i\}_{i\in I}\in \dgalg_k^I.$

Then there exists a homotopy limit
\begin{equation}\cA=\holim_{\substack{I}}\cA_i.\end{equation}

We would like to write it in explicit form.
\begin{defi}For a morphism $s:x\to y$ in the category $I,$ we put $r(s):=y,$ $l(s)=x.$ We denote by $\overline{\Mor(I)}$ the set of non-identical morphisms in $I.$\end{defi}

We put
\begin{equation}\cA^i=\prod_{\substack{s_1,\dots,s_p\in\overline{\Mor(I)}, p>0,\\l(s_{i+1})=r(s_i)}}\cA_{r(s_p)}^{i-p}\times\prod_{\substack{x\in I}}\cA_x^i.\end{equation}
For $a\in \cA,$ we denote by $a_{s_p,\dots,s_1}\in\cA_{r(s_p)},$ $a_x\in\cA_x$ the corresponding components. It is convenient to consider components $a_x$ to be corresponding to empty paths
in $I,$ with final object $x$. With this in mind, the differential and the composition are defined as follows. For homogeneous $a,b\in \cA,$
\begin{multline}d(a)_{s_p,\dots,s_1}=d(a_{s_p,\dots,s_1})+(-1)^{\bar{a}+1} s_p(a_{s_{p-1},\dots,s_1})+\\
\sum\limits_{i=1}^{p-1}(-1)^{\bar{a}+1+p-j} a_{s_p,\dots,s_{i+1}s_i,\dots,s_1}+(-1)^{\bar{a}+1+p} a_{s_p,\dots,s_2},\end{multline}
\begin{equation}(a\cdot b)_{s_p,\dots,s_1}=\sum\limits_{i=0}^p (-1)^{(p-i)\bar{b}} a_{s_p,\dots,s_{i+1}}\cdot s_p\dots s_{i+1}(b_{s_i,\dots,s_1}),\end{equation}
where $\bar{a}$ (resp. $\bar{b}$) denote the degree of $a$ (resp. $b$).

Now suppose that we have a compatible system of morphisms $f_x:\cB\to \cA_x,$ $x\in I,$ in $\dgalg_k$ (i.e. $sf_x=f_y$ for $s:x\to y$). Then we have natural morphism
$f:\cB\to \cA=\holim_{\substack{I}}\cA_i,$ given by the formula
\begin{equation}\begin{cases}f(b)_x=f_x(b) & \text{for }x\in I;\\
f(b)_{s_p,\dots,s_1}=0 & \text{for }p>0.\end{cases}\end{equation}

Now suppose that we have also a functor $I^{op}\to Z^0(\cC\text{-Mod}),$ $x\to M_x,$ where $\cC$ is some DG category, and $Z^0(\cC\text{-Mod})$ is the abelian category of right DG $\cC$-modules. Then there exists a homotopy colimit
\begin{equation}M=\hocolim_{\substack{I^{op}}}M_x.\end{equation}
Again, we can write $M$ explicitly as follows:
\begin{equation}M(X)^i=\bigoplus_{\substack{s_1,\dots,s_p\in\overline{\Mor(I)}, p>0,\\l(s_{i+1})=r(s_i)}}M_{r(s_p)}(X)^{i+p}\oplus\bigoplus_{\substack{x\in I}}M_x(X),\quad X\in\cC.\end{equation}

For $m\in M_{r(s_p)}(X)$ (resp. $m\in M_x(X)$) we denote by $m_{s_p,\dots,s_1}\in M(X)$ (resp. $m_x\in M(X)$) the corresponding elements with only one component. Again,
it is convenient to consider $m_x$ to be corresponding to an empty path
in $I,$ with final object $x$. For a homogeneous $m,$ we have that $\deg(m_{s_p,\dots,s_1})=\deg(m)-p.$ For a homogeneous $m_{s_p,\dots,s_1},$ we have
\begin{multline}d(m_{s_p,\dots,s_1})=d(m)_{s_p,\dots,s_1}+(-1)^{\bar{m}} s_p(m)_{s_{p-1},\dots,s_1}+\\
\sum\limits_{i=1}^{p-1}(-1)^{\bar{m}+p-i}m_{s_p,\dots,s_{i+1}s_i,\dots,s_1}+(-1)^{\bar{m}+p} m_{s_p,\dots,s_2}.\end{multline}
Further, for a homogeneous $f\in\Hom_{\cC}(Y, X),$ we have
\begin{equation}m_{s_p,\dots,s_1}\cdot f=(-1)^{p\bar{f}}(mf)_{s_p,\dots,s_1}.\end{equation} 

Suppose that we have a compatible system of morphisms $g:M_x\to N$ for some DG module $N$ (i.e. $g_xs=g_y$ for $s\in\Hom_I(x, y)$). Then we have natural morphism
$g:M=\hocolim_{\substack{I^{op}}}M_x\to N,$ given by the formula
\begin{equation}\begin{cases}g(m_x)=g_x(m) & \text{for }x\in I;\\
f(m_{s_p,\dots,s_1})=0 & \text{for }p>0.\end{cases}\end{equation}

Now, suppose that, with the above notation, we have a system of morphisms of DG algebras $\varphi_x:\cA_x^{op}\to \End_{\cC}(M_x),$ $x\in I,$ which are compatible in the following sense:
\begin{equation}\varphi_x(a)(s(m))=s(\varphi_y(s(a))(m)),\quad a\in \cA_x,\quad m\in M_y,\quad s\in\Hom_I(x,y).\end{equation}
Then we have a natural morphism
\begin{equation}\label{lim_on_colim}\cA^{op}=(\holim_{\substack{I}}\cA_x)^{op}\to \End_{\cC}(M)=\End_{\cC}(\hocolim_{\substack{I^{op}}}M_x).\end{equation}

Explicitly, for homogeneous $a\in\cA,$ $m_{s_p,\dots,s_1}\in M(X),$ we have
\begin{equation}a(m_{s_p,\dots,s_1})=\sum\limits_{i=0}^p (-1)^{i(p-i+\bar{a})} (s_p\dots s_{i+1})(\varphi_{r(s_p)}(a_{s_p,\dots,s_{i+1}})(m))_{s_{i},\dots,s_1}.\end{equation}

Now we are ready to formulate and prove our main technical result.

\begin{lemma}\label{homot_limits}Let $\cA$ be a DG algebra, $\cT\subset D(\cA)$ a full thick essentially small triangulated subcategory. Suppose that $I$ is a small category,
$\{\cA_x\}_{x\in I}\in \dgalg_k^I,$ and we have a compatible system of morphisms $f:\cA\to \cA_x,$ $x\in I.$ Assume that all $\cA_x$ lie in $\cT$ as right
DG $\cA$-modules, and for any $E\in\cT$
the natural map \begin{equation}\hocolim_{\substack{I^{op}}}\bR\Hom_{\cA}(\cA_x,E)\to E\end{equation}
is an isomorphism in $D(k).$ Then we have natural commutative diagram in $\Ho(\dgalg_k):$
\begin{equation}
\label{commut}
\begin{CD}
\cA @>\id >> \cA\\
@V\iota_{\cT} VV                              @VVV\\
\widehat{\cA}_{\cT} @>\cong >> \holim_{\substack{I}}\cA_x.
\end{CD}
\end{equation}
\end{lemma}

\begin{proof}Choose some set of h-injective $\cA^{op}$-modules which generate $\cT,$ and denote by $\cD$ the corresponding DG category. Then
by our assumptions, we have natural quasi-isomorphism of DG $\cD$-modules:
\begin{equation}\hocolim_{\substack{I^{op}}}\Hom_{\cA}(\cA_x,-)\to \Hom_{\cA}(\cA,-).\end{equation}
Therefore, we have natural isomorphism in $\Ho(\dgalg):$
\begin{equation}\label{iso}\widehat{\cA}_{\cT}^{op}\cong \bR\End_{\cD^{op}}(\hocolim_{\substack{I^{op}}}\Hom_{\cA}(\cA_x,-)).\end{equation}
We have natural compatible system of morphisms of DG algebras:
\begin{equation}\varphi_x:\cA_x^{op}\to\End_{\cD^{op}}(\Hom_{\cA}(\cA_x,-)).\end{equation}
Therefore, as in \eqref{lim_on_colim}, we have natural morphism
\begin{equation}\varphi:(\holim_{\substack{I}}\cA_x)^{op}\to\End_{\cD^{op}}(\hocolim_{\substack{I^{op}}}\Hom_{\cA}(\cA_x,-)).\end{equation} Composing it with natural map from $\End$ to
$\bR\End$ (in $\Ho(\dgalg_k)$) and \eqref{iso}, we obtain a natural morphism
\begin{equation}\label{desired}\holim_{\substack{I}}\cA_x\to \widehat{\cA}_{\cT}\end{equation}
in $\Ho(\dgalg).$ Further, since $\cA_x\in\cT,$ we have natural isomorphisms in $D(k):$
\begin{equation}\cA_x\stackrel{\sim}{\to}\bR\Hom_{\cD^{op}}(\Hom_{\cA}(\cA_x,-),\Hom_{\cA}(\cA,-)).\end{equation}
To conclude that \eqref{desired} is an isomorphism, it suffices to note the following chain of isomorphisms in $D(k):$
\begin{multline}\label{chain}\widehat{\cA}_{\cT}\cong \bR\End_{\cD^{op}}(\Hom_{\cA}(\cA,-))\cong\\
\bR\Hom_{\cD^{op}}(\hocolim_{\substack{I^{op}}}\Hom_{\cA}(\cA_x,-),\Hom_{\cA}(\cA,-))\cong\\
\holim_{\substack{I}}\bR\Hom_{\cD}(\Hom_{\cA}(\cA_x,-),\Hom_{\cA}(\cA,-))
\cong \holim_{\substack{I}}\cA_x.\end{multline}

It is easy to check that the composition \eqref{chain} is inverse (in $D(k)$) to the morphism of DG algebras \eqref{desired}, so we obtain the desired isomorphism in $\Ho(\dgalg).$
Commutativity of $\eqref{commut}$ is straightforward to check.
\end{proof}

\subsection{Algebraizable derived categories of formal completions of schemes }
\label{algebraizable}

Let $X$ be a separated Noetherian $k$-scheme. Recall that \cite{BvdB} $D(X)=D(\QCoh(X)),$ the derived category of quasi-coherent sheaves on $X,$ is compactly generated by
one object, and $D(X)^c=\Perf(X).$ More precisely, they prove this for the category $D_{qch}(X)$ of complexes of $\cO_X$-modules with quasi-coherent
cohomology, but for $X$ separated the latter category is known to be equivalent to $D(\QCoh(X))$ (see \cite{BvdB}).

Now let $Y\subset X$ a closed subscheme.
We would like to define the algebraizable derived category $D_{alg}(\widehat{X}_Y).$

Let $\cI_Y\subset\cO_X$ be ideal sheaf defining $Y.$ Denote by $Y_n\subset X$ the $n$-th infinitesimal neighborhood of $Y,$ with ideal sheaf $\cI_Y^n.$
Denote by $\iota_{n,n+1}:Y_n\to Y_{n+1},$ $\iota_n:Y_n\to X$ the natural inclusions. Choose some DG enhancements
for $\Perf(X)$ and $\Perf(Y_n),$ with DG enhancements of functors $\bL\iota_n^*,$ $\bL\iota_{n,n+1}^*$ (we write the corresponding DG functors in the same way),
so that we have equalities of DG functors $\bL\iota_n^*=\bL\iota_{n,n+1}^*\bL\iota_{n+1}^*.$ Denote by $\bR\Hom(-,-)$ the complexes of morphisms in the corresponding DG enhancements.

Define the DG category $\Perf_{alg}(\widehat{X}_Y)$ as follows. Its objects are the same as in $\Perf(X).$ Further, for $\cE,\cF\in\Perf(X),$ we put
\begin{equation}\Hom_{\Perf_{alg}(\widehat{X}_Y)}(\cE,\cF):=\holim_n\bR\Hom(\bL\iota_n^*\cE,\bL\iota_n^*\cF).\end{equation}
Composition are defined in the obvious way (as in the case of homotopy limits of DG algebras). Define algebraizable derived category by the formula
\begin{equation}D_{alg}(\widehat{X}_Y):=D(\Perf_{alg}(\widehat{X}_Y)).\end{equation}

We have an obvious DG functor
\begin{equation}\kappa^*:\Perf(X)\to\Perf_{alg}(\widehat{X}_Y),)\end{equation}
and the corresponding functor
\begin{equation}\bL\kappa^*:D(X)\to D_{alg}(\widehat{X}_Y).\end{equation}

\begin{remark}In the case when $X=\Spec(A)$ is affine, and $Y=\Spec(A/I),$ we easily see that
\begin{equation}D_{alg}(\widehat{X}_Y)\cong D(\widehat{A}_I),\end{equation}
where $\widehat{A}_I=\lim_n A/I^n$ is the $I$-adic completion of $A.$ The functor $\kappa^*$ in this case is just the restriction of scalars for the natural morphism $A\to \widehat{A}_I.$ \end{remark}

\begin{theo}\label{main-theo}Let $X$ be a separated Noetherian scheme, and $Y\subset X$ a closed subscheme. Then we have the following commutative diagram:
\begin{equation}
\label{commut2}
\begin{CD}
D(X) @>\id >> D(X)\\
@VVV                              @V\bL\kappa^* VV\\
\widehat{D(X)}_{D^b_{coh,Y}(X)} @>\cong >> D_{alg}(\widehat{X}_Y).
\end{CD}
\end{equation}
\end{theo}

\begin{proof} We follow notation above the theorem. Choose a generator $\cF\in\Perf(X).$ Then $\kappa^*(\cF)\in \Perf_{alg}(\widehat{X}_Y)$ is a compact generator of $D_{alg}(\widehat{X}_Y).$ Put
\begin{equation}\cA:=\bR\End(\cF),\quad \cA_n:=\bR\End(\bL\iota_n^*\cF).\end{equation}
We have obvious morphisms $\bL\iota_{n,n+1}^*:\cA_{n+1}\to\cA_n.$ Hence $\{\cA_n\}_{n\in \N}\in\dgalg_k^{\N^{op}},$ where we treat $\N$ as a category: $Ob(\N)=\Z_{>0},$
$\Hom_{\N}(i,j)=\emptyset$ for $i>j,$ and there is exactly one morphism $i\to j$ for $i\leq j.$ Further, we have morphisms of DG algebras
\begin{equation}\label{morphisms}\bL\iota_n^*:\cA\to\cA_n,\end{equation}
which are compatible with $\bL\iota_{n,n+1}^*$ by our assumptions. Further, denote by $\cT\subset D(\cA)$ the essential image of $D^b_{coh,Y}(X)$ under the equivalence
\begin{equation}\bR\Hom(\cF,-):D(X)\to D(\cA).\end{equation}
By adjunction, $\cA_n\cong \bR\Hom(\cF,\iota_{n*}\bL\iota_n^*\cF)\in\cT.$ We claim that the data of $\cA,$ $\{\cA_n\}_{n\in \N}\in\dgalg_k^{\N^{op}},$ $\cT\subset D(\cA)$ and morphisms \eqref{morphisms}
satisfies the assumptions of Lemma \ref{homot_limits}. Indeed, by Grothendieck Theorem \cite{Gr}, for any $\cG\in \Perf(X),$ $E\in D^Y(X)$ we have an isomorphism in $D(X):$
\begin{equation}\hocolim_n \bR\cHom(\iota_{n*}\bL\iota_n^*\cG,E)\stackrel{\sim}{\to} \bR\cHom(\cG,E).\end{equation}
Moreover, since the functor $\bR\Gamma$ commutes with infinite direct sums, we have a chain of isomorphisms in $D(k):$
\begin{multline}\hocolim_n \bR\Hom(\iota_{n*}\bL\iota_n^*\cG,E)\cong \hocolim_n \bR\Gamma(\bR\cHom(\iota_{n*}\bL\iota_n^*\cG,E))\cong\\
\bR\Gamma(\hocolim_n \bR\cHom(\iota_{n*}\bL\iota_n^*\cG,E))\cong \bR\Gamma(\bR\cHom(\cG,E))\cong \bR\Hom(\cG,E).\end{multline}
Therefore, we have the following isomorphisms:
\begin{multline}\hocolim_n\bR\Hom_{\cA}(\cA_n,\bR\Hom(\cF,E))\cong \hocolim_n \bR\Hom(\iota_{n*}\bL\iota_n^*\cF,E)\\
\cong \bR\Hom(\cF,E).\end{multline}
Hence, the assumptions of Lemma \ref{homot_limits} are satisfied. Applying it, we obtain the chain of equivalences
\begin{equation}\widehat{D(X)}_{D^b_{coh,Y}(X)}\cong\widehat{D(\cA)}_{\cT}\cong D(\holim_n\cA_n)\cong D_{alg}(X).\end{equation}
The last equivalence follows from the observation that $\kappa^*(\cF)\in\Perf_{alg}(\widehat{X}_Y)$ is a compact generator of $D_{alg}(X)$ and
\begin{equation}\End_{\Perf_{alg}(\widehat{X}_Y)}(\kappa^*(\cF))\cong\holim_n\cA_n.\end{equation}
Commutativity of \eqref{commut2} is straightforward.
\end{proof}

\begin{cor}Let $X$ be a separated Noetherian scheme, and $Y\subset X$ a closed subscheme. Then

1) The restriction of the functor $\bL\kappa^*:D(X)\to D_{alg}(\widehat{X}_Y)$ to $D^b_{coh,Y}(X)$ is full and faithful;

2) The functor
\begin{equation}D_{alg}(\widehat{X}_Y)\to \widehat{D_{alg}(\widehat{X}_Y)}_{D^b_{coh,Y}(X)}\end{equation}
is an equivalence.\end{cor}

\begin{proof}Recall that the category $D_Y(X)$ is compactly generated by $\Perf_Y(X)\subset D^b_{coh,Y}(X)$ \cite{AJPV} Hence, both 1) and 2) are direct consequences of Theorems \ref{main-theo} and \ref{completion_complete}.\end{proof}

We have a nice corollary for completions of regular Noetherian $k$-algebras.

\begin{cor}\label{regular}Let $R$ be a regular commutative Noetherian $k$-algebra, and $M\in D_{f.g.}(R)\cong D^b_{coh}(\Spec R)$
be a complex of $R$-modules with finitely generated cohomology. Denote by $I\subset R$ the annihilator of $H^{\cdot}(M),$
so that $V(\sqrt{I})\subset \Spec R$ is precisely the support of $M.$ Then we have an isomorphism
\begin{equation}\widehat{R}_M\cong \widehat{R}_I,\end{equation}
where the RHS is the ordinary $I$-adic completion.\end{cor}

\begin{proof}By a Theorem of Hopkins \cite{Ho} and Neeman \cite{Nee}, all full thick triangulated subcategories of $\Perf(R)\cong \Perf(\Spec R)$
generated by one object
are of the form $\Perf_Z(\Spec R)$ (perfect complexes with cohomology supported on $Z$) for a closed subset $Z\subset \Spec R.$ Further,
since $R$ is regular, we have that $D^b_{coh}(\Spec R)=\Perf(\Spec R).$ It follows that $M$ generates the subcategory $D^b_{coh,V(\sqrt{I})}(\Spec R)\subset D^b_{coh}(\Spec R).$ It remains to apply Theorem \ref{main-theo}.\end{proof}

\begin{cor}\label{surj_of_alg}Let $R$ be commutative Noetherian $k$-algebra, and $I\subset R$ an ideal. Assume that either $R$ or $R/I$ is regular. Then we have an isomorphism
\begin{equation}\widehat{R}_{(R/I)}\cong\widehat{R}_I,\end{equation}
where the RHS is ordinary $I$-adic completion.\end{cor}

\begin{proof} If $R$ is regular, the isomorphism follows from Corollary \ref{regular}. Assume that $R/I$ is regular.

Put $X:=\Spec(R),$ and $Y:=\Spec(R/I)\subset X.$ We claim that $\iota_*\cO_Y$ is a generator of $D^b_{coh,Y}(X).$
Indeed, $\iota_*\cO_Y$ generates all objects $\iota_*\cF,$ $\cF\in D^b_{coh}(Y),$ which generate the whole subcategory $D^b_{coh,Y}(X).$

Therefore, the assertion follows from Theorem \ref{main-theo}.
\end{proof}

The following Proposition shows that in the Corollary \ref{surj_of_alg} one cannot drop the assumption of regularity.

\begin{prop}\label{infin_extensions}1) Let $R$ be some commutative algebra over a field $k,$ and $M$ an $R$-module. Denote by $\tilde{R}$ the split square-zero extension of
$R$ by $M.$ The following are equivalent:

(i) The map $\tilde{R}\to \widehat{\tilde{R}}_{R}$ is an isomorphism in $\Ho(\dgalg_k);$

(ii) The following are isomorphisms in $D(R):$
\begin{equation}\label{MtoM**}M\stackrel{\sim}{\to} M^{\vee\vee},\end{equation}
\begin{equation}\label{tensor_powers}(M^{\stackrel{\bL}{\otimes}n})^{\vee\vee}
\stackrel{\sim}{\to} ((M^\vee)^{\stackrel{\bL}{\otimes}n})^{\vee},\quad n\geq 2.\end{equation}
Here tensor products are over $R$ and $(-)^\vee=\bR\Hom_R(-,R).$

2) In particular, if $R=k[x]/(x^2)$ and $M=k,$ then the map $\tilde{R}\to \widehat{\tilde{R}}_{R}$
is not an isomorphism.\end{prop}

\begin{proof} 1) Let $A$ be any DG algebra, and $N$ a DG $A$-module. Then we can treat $A$ as an $A_{\infty}$-algebra, and
$N$ as a right $A_{\infty}$-module over $A.$ Denote by $A\text{-mod}_{\infty}$ the DG category of right $A_{\infty}$-modules
over $A.$ We put
\begin{equation}B_N:=\End_{A\text{-mod}_{\infty}}(N)=\prod\limits_{n\geq 0}\Hom_k(N\otimes A^{\otimes n}, N)[-n].\end{equation}
Further, we have obvious projection morphism of DG algebras
$B_N\to \End_k(N),$ hence $N$ is naturally a DG module over $B_N^{op}.$ We put
\begin{equation}\widehat{A}_M:=(\End_{B_M^{op}\text{-mod}_{\infty}}(M))^{op}.\end{equation}
Then $\widehat{A}_M$ is a DG model for derived double centralizer of $M.$ We have a natural $A_{\infty}$-morphism $A\to\widehat{A}_M.$

Now put $A:=\tilde{R},$ and $N:=R.$ Then DG algebra $B_N$ as a complex can  be decomposed into the product of complexes:
\begin{multline}B_M=\prod\limits_{n\geq 0}C_n,\\ C_n:=\prod\limits_{l_1>0,l_2,\dots,l_{n+1}\geq 0}
\Hom_k(R^{\otimes l_1}\otimes M\otimes\dots\otimes M\otimes R^{\otimes l_{n+1}},R)[-l_1-\dots-l_{n+1}-n+1].\end{multline}

Further, the DG algebra $\widehat{\tilde{R}}_R$ as a complex can be decomposed into the product of complexes:
\begin{equation}\widehat{\tilde{R}}_R=\prod\limits_{n\geq 0}D_n,\quad D_n:=\prod_{\substack{m_1+\dots+m_l=n,\\
m_i\geq 0}} \Hom_k(C_{m_1}\otimes\dots\otimes C_{m_l}\otimes R,R)[-l].\end{equation}

It is straightforward to observe the following isomorphisms in $D(k):$
\begin{equation}D_0\cong R,\quad D_1\cong M^{\vee\vee}.\end{equation}

Thus, (i) holds iff the map \eqref{MtoM**} is an isomorphism, and all the complexes $D_n,$ $n\geq 2,$ are acyclic. Further, it is straightforward
to show by
induction on $m\geq 2$ that the following are equivalent:

(iii) the map \eqref{MtoM**} and \eqref{tensor_powers} for $2\leq n\leq m$ are isomorphisms;

(iv)  the map \eqref{MtoM**} is an isomorphism and the complexes $D_n,$ $2\leq n\leq m,$ are acyclic.

This proves part 1) of Proposition.

2) We claim that in the case $R=k[x]/(x^2)$ and $M=k$ the map
\begin{equation}(M\stackrel{\bL}{\otimes}_R M)^{\vee\vee}\to (M^{\vee}\stackrel{\bL}{\otimes}_R M^{\vee})^{\vee}\end{equation}
is not an isomorphism. Indeed, we have isomorphisms in $D(R):$
\begin{equation}(M\stackrel{\bL}{\otimes}_R M)^{\vee\vee}\cong\bigoplus\limits_{n\geq 0}k[n],\quad (M^{\vee}\stackrel{\bL}{\otimes}_R M^{\vee})^{\vee}\cong \bigoplus\limits_{n\geq 0}k[-n].\end{equation}
Therefore, according to 1), the map $\tilde{R}\to \widehat{\tilde{R}}_{R}$
is not an isomorphism. Proposition is proved.
\end{proof}

\section{Beilinson-Parshin adeles and categorical formal completions.}
\label{sect_adeles}

Let $X$ be a separated Noetherian $k$-scheme of finite Crull dimension $d.$

We first recall reduced Beilinson-Parshin adeles of $X$ \cite{Be, P}. Denote by $P(X)$ the set of all schematic points of $X.$ Put
\begin{equation}S(X)_p^{red}:=\{(\eta_0,\dots,\eta_p):\eta_i\in P(X),\eta_i\ne\eta_{i-1},\eta_i\in\overline{\eta_{i-1}}\text{ for }1\leq i\leq p\}\end{equation}
For $T\subset S(X)_p,$ $p>0,$ and $\eta\in P(X),$ put
\begin{equation}T_{\eta}:=\{(\eta_1,\dots,\eta_p)\in S(X)_{p-1}:(\eta,\eta_1,\dots,\eta_p)\in T\}\subset S(X)_{p-1}.\end{equation}
Also, for $\eta\in P(X),$ denote by $j_{\eta}:\Spec(\cO_{\eta})\to X$ the natural map. Denote by $\m_{\eta}\subset\cO_{\eta}$ the unique maximal ideal.

For each subset $T\subset S(X)_p,$ $0\leq p\leq d=\dim X,$ we will define a functor
\begin{equation}\label{adeles}\A_T(X,-):\QCoh(X)\to k\text{-Mod},\end{equation}
exact and commuting with infinite direct sums (hence commuting with small colimits). Since each quasi-coherent sheaf is a union of its coherent subsheaves, it suffices to define the functor $\A_T(X,-)$ for coherent sheaves.

We define these functors by induction on $p.$ For $p=0,$ $T\subset S(X)_0=P(X),$ and $\cF\in\Coh(X),$ we put
\begin{equation}\A_T(X,\cF):=\prod_{\substack{\eta\in T}}\widehat{\cF}_{\eta}.\end{equation}
With above said, this defines uniquely the functor \eqref{adeles} for $p=0.$ It is easy to check that it is exact and commutes with small colimits.

Now let $T\subset S(X)_p,$ $p>0.$ Suppose that all the functors $\A_{T_{\eta}}(X,-)$ are already defined. For $\cF\in \Coh(X),$ put
\begin{equation}\A_T(X,\cF):=\prod_{\substack{\eta\in P(X)}}\lim_n\A_{T_{\eta}}(j_{\eta*}(\cF_{\eta}/\m_{\eta}^n)).\end{equation}
This defines uniquely the functor \eqref{adeles} for all $T\subset P(X),$ $p>0,$ and by induction we see that $\A_T(X,-)$ is exact and commutes with small colimits.

For $d\geq k_0>\dots>k_p\geq 0,$ we put
\begin{equation}S(X)_{(k_0,\dots,k_p)}:=\{(\eta_0,\dots,\eta_p)\in S(X)_p:\dim\overline{\eta_i}=k_i\text{ for }0\leq i\leq p\}.\end{equation}
Further, put
\begin{equation}\A(X,-)_p:=\A_{S(X)_p}(X,-),\quad\A(X,-)_{(k_0,\dots,k_p)}:=\A_{S(X)_{(k_0,\dots,k_p)}}(X,-).\end{equation}
Clearly, we have
\begin{equation}\A(X,-)_p\cong\prod_{\substack{d\geq k_0>\dots>k_p\geq 0}}\A(X,-)_{(k_0,\dots,k_p)}.\end{equation}

It is easy to see that for all $T\subset S(X)_p,$ the $k$-module $\A_T(X,\cO_X)$ is naturally a commutative $k$-algebra. Further, for all quasi-coherent $\cF$
the $k$-module $\A_T(X,\cF)$ is naturally an $\A_T(X,\cO_X)$-module. For convenience, we put
\begin{equation}\A(X)_p:=\A(X,\cO_X)_p,\quad\A(X)_{(k_0,\dots,k_p)}:=\A(X,\cO_X)_{(k_0,\dots,k_p)}.\end{equation}
To formulate main result of this section, we would like to use the following notation. If $F:\cT_1\to \cT_2$ is a functor between compactly generated
triangulated categories,
and $S\subset \cT_1$ is an essentially small Karoubian complete triangulated subcategory (resp. localizing subcategory), then we put
\begin{equation}\widehat{\cT_2}_S:=\widehat{\cT_2}_{\langle F(S)\rangle},\quad(\text{ rep. }\cT_2/S:=\cT_2/\langle F(S)\rangle),\end{equation}
where $\langle F(S)\rangle$ is subcategory classically generated by $F(S)$ (resp. smallest localizing subcategory containing $F(S)$).

Denote by $D^b_{coh,\leq p}(X)\subset D^b_{coh}(X)$ the full subcategory consisting of complexes, for which the dimension of support of cohomology is not greater than $p.$ Further, Denote by $D_{\leq p}(X)\subset D(X)$ the smallest localizing subcategory, which contains $D^b_{coh,\leq k}(X).$ It is
clear that $D_{\leq k}(X)$ is compactly generated by $\Perf_{\leq k}(X)=\Perf(X)\cap D^b_{coh,\leq k}(X).$

\begin{theo}\label{relate_to_adeles}Let $X$ be a separated Noetherian $k$-scheme of dimension $d.$ Fix some sequence $d\geq k_0>\dots>k_p\geq 0.$ If $p=0,$ then we have a commutative diagram,
\begin{equation}
\label{commut3}
\begin{CD}
D(X) @>\id >> D(X)\\
@V\A(X,-)_{(k_0)}VV                              @VVV\\
D(\A(X)_{(k_0)}) @>\cong >> \widehat{(D(X)/D_{\leq(k_0-1)}(X))}_{D^b_{coh,\leq k_0}}.
\end{CD}
\end{equation}

 For $p>0,$ there is a natural commutative diagram
\begin{equation}
\label{commut4}
\begin{CD}
D(\A(X)_{(k_1,\dots,k_p)}) @>\id >> D(\A(X)_{(k_1,\dots,k_p)})\\
@VVV                              @VVV\\
D(\A(X)_{(k_0,\dots,k_p)}) @>\cong >> \widehat{(D(\A(X)_{(k_1,\dots,k_p)})/D_{\leq(k_0-1)}(X))}_{D^b_{coh,\leq k_0}}.
\end{CD}
\end{equation}
\end{theo}

\begin{proof}First we prove \eqref{commut3}.

\begin{lemma}\label{description}1) The functor
\begin{equation}\label{right_adj}D(X)_{\leq k_0}\to D(X)_{\leq k_0},\quad\cF\mapsto\bigoplus\limits_{\eta\in S(X)_{(k_0)}}j_{\eta*}(\cF_{\eta}),\end{equation}
is the projection onto the right orthogonal to $D_{\leq(k_0-1)}(X).$ In particular, it induces a fully faithful embedding of $D(X)_{\leq k_0}/D_{\leq(k_0-1)}(X)$ into $D(X).$

2) The images of $\cF\in D^b_{coh,\overline{\eta_0}}(X)$ in $D(X)/D_{\leq(k_0-1)}(X),$ where $\eta_0\in S(X)_{(k_0)},$ generate the essential image of $D^b_{coh,\leq k_0}(X).$

3) If $\eta_1,\eta_2\in S(X)_{(k_0)},$ $\eta_1\ne \eta_2,$ and $\cF_i\in D^b_{coh,\overline{\eta_i}}(X)$ for $i=1,2,$ then we have
\begin{equation}\Hom_{D(X)/D_{\leq(k_0-1)}(X)}(\cF_1,\cF_2)=0.\end{equation}\end{lemma}

\begin{proof}1) By adjunction, for any $\eta\in S(X)_{(k_0)}$ and $\cF\in D(X)_{\leq k_0}$ we have that $j_{\eta*}(\cF_{\eta})$ is
 right orthogonal to $D(X)_{\leq(k_0-1)}.$ It remains to note that the cone of the natural morphism
 \begin{equation}\cF\to \bigoplus\limits_{\eta\in S(X)_{(k_0)}}j_{\eta*}(\cF_{\eta}),\quad\cF\in D(X)_{\leq k_0},\end{equation}
 lies in $D_{\leq(k_0-1)}(X).$

2) We have that $D^b_{coh,\leq k_0}(X)$ is itself generated by $D^b_{coh,\overline{\eta_0}}(X),$ $\eta_0\in S(X)_{(k_0)}.$ This implies 2).

3) follow from 1) easily.\end{proof}

 Let $E\in\Perf(X)$ be a generator. Put $\cA:=\bR\Hom_{D(X)/D_{\leq (k_0-1)}(X)}(E,E).$ Let $\cT\subset D(\cA)$ (resp. $\cT_{\eta}\subset D(\cA)$, $\eta\in S(X)_{(k_0)}$) be the subcategory
classically generated by the image of $D^b_{coh,\leq k_0}(X)$ (resp. $D^b_{coh,\overline{\eta}}$). Then, by Lemma \ref{description} and Proposition \ref{product_gen} we have that
\begin{equation}\label{product}\widehat{\cA}_{\cT}\cong\prod\limits_{\eta\in S(X)_{(k_0)}}\widehat{\cA}_{\cT_{\eta}}.\end{equation}
Fix some $\eta\in S(X)_{(k_0)}$ and put $Y:=\overline{\eta}.$ We claim that $\widehat{\cA}_{\cT_{\eta}}\cong\bR\End_{\widehat{\cO}_{\eta}}(\widehat{E}_{\eta}).$ This can be shown as follows. Denote by $Y_l$ the infinitesimal neighborhoods of $Y,$ $\iota_l:Y_l\to X$ the inclusions, and $j_{\eta,l}:\Spec(\cO_{\eta}/\m_{\eta}^l)\to X$ natural morphisms. Put
$\cA_l:=\bR\End(\bL j_{\eta,l}^*E).$ Then $\{\cA_l\}_{l\in\N}\in\dgalg_k^{\N^{op}},$ and we have a compatible system of morphisms $\cA\to \cA_l.$ Further, there are isomorphisms
\begin{equation}\cA_l\cong \bR\Hom_{D(X)}(E,j_{\eta,l*}\bL j_{\eta,l}^*E)\cong\bR\Hom_{D(X)/D_{\leq(k_0-1)}(X)}(E,\iota_{l*}\bL\iota_l^*E)\text{ in }D(\cA).\end{equation}
For each $\cF\in D^b_{coh,Y}(X),$ we have the following chain of isomorphisms:
\begin{multline}\hocolim_l\bR\Hom_{\cA}(\cA_l,\bR\Hom_{D(X)/D_{\leq(k_0-1)}(X)}(E,\cF))\cong\\
\hocolim_l\bR\Hom_{D(X)/D_{\leq(k_0-1)}(X)}(\iota_{l*}\bL\iota_l^*E,\cF)\cong \hocolim_l\bR\Hom_{D(X)}(\iota_{l*}\bL\iota_l^*E,j_{\eta*}(\cF_{\eta}))\cong\\
\bR\Hom_{D(X)}(E,j_{\eta*}(\cF_{\eta}))\cong \bR\Hom_{D(X)/D_{\leq(k_0-1)}(X)}(E,\cF).\end{multline}
Hence, by Lemma \ref{homot_limits}, we have
\begin{equation}\widehat{\cA}_{\cT_{\eta}}\cong\holim_l\bR\End(\bL j_{\eta,l}^*E)\cong\bR\End_{\widehat{\cO}_{\eta}}(\widehat{E}_{\eta}).\end{equation}
According to \eqref{product}, we have that
\begin{equation}\widehat{\cA}_{\cT}\cong\prod\limits_{\eta\in S(X)_{(k_0)}}\bR\End_{\widehat{\cO}_{\eta}}(\widehat{E}_{\eta}).\end{equation}
Further, since $\widehat{E}_{\eta},\widehat{\cO}_{\eta}\in\Perf(\widehat{\cO}_{\eta})$ generate each other in uniformly bounded number of steps,
we have Morita equivalence
\begin{equation}D(\prod\limits_{\eta\in S(X)_{(k_0)}}\bR\End_{\widehat{\cO}_{\eta}}(\widehat{E}_{\eta}))\cong D(\prod\limits_{\eta\in S(X)_{(k_0)}}\widehat{\cO}_{\eta}).\end{equation}
Further, by definition, $\prod\limits_{\eta\in S(X)_{(k_0)}}\widehat{\cO}_{\eta}=\A(X)_{(k_0)}.$ Hence, we have equivalences
\begin{equation}\widehat{(D(X)/D_{\leq(k_0-1)}(X))}_{D^b_{coh,\leq k_0}}\cong D(\widehat{\cA}_{\cT})\cong D(\A(X)_{(k_0)}),\end{equation}
and commutativity of \eqref{commut3} is straightforward.

Now we prove \eqref{commut4}. We have the morphisms of algebras
\begin{equation}g_{\eta}:\A(X)_{(k_1,\dots,k_p)}\to \A(X,j_{\eta*}\cO_{\eta})_{(k_1,\dots,k_p)},\quad\eta\in S(X)_{(k_0)}.\end{equation}
Note that we have the following isomorphisms of functors
\begin{equation}\A(X,j_{\eta*}(-))_{(k_1,\dots,k_p)}\cong g_{\eta*}(\A(X,j_{\eta*}\cO_{\eta})_{(k_1,\dots,k_p)}\stackrel{\bL}{\otimes}_{\cO_{\eta}}-),\end{equation}
\begin{equation}\bL g_{\eta}^*\A(X,-)_{(k_1,\dots,k_p)}\cong \A(X,j_{\eta*}\cO_{\eta})_{(k_1,\dots,k_p)}\stackrel{\bL}{\otimes}_{\cO_{\eta}}\bL j_{\eta}^*(-).\end{equation}

\begin{lemma}\label{description2}1) If $\cF\in D_{\leq(k_0-1)}(X),$ $\cG\in D_{\leq k_0}(X),$ and $\eta\in S(X)_{(k_0)},$ then
\begin{equation}\Hom_{\A(X)_{(k_1,\dots,k_p)}}(\A(X,\cF)_{(k_1,\dots,k_p)},\A(X,j_{\eta*}\cG_{\eta})_{(k_1,\dots,k_p)})=0.\end{equation}

2) The objects $\A(X,\cF)_{(k_1,\dots,k_p)},$ where $\cF\in D^b_{coh,\overline{\eta_0}}(X),$ $\eta_0\in S(X)_{(k_0)},$ generate the essential image of $D^b_{coh,\leq k_0}(X)$ in $D(\A(X)_{(k_1,\dots,k_p)})/D_{\leq(k_0-1)}(X).$

3) If $\eta_1,\eta_2\in S(X)_{(k_0)},$ $\eta_1\ne\eta_2,$ and $\cF_i\in D^b_{coh,\overline{\eta_i}}(X)$ for $i=1,2,$ then we have
\begin{equation}\Hom_{D(\A(X)_{(k_1,\dots,k_p)})/D_{\leq(k_0-1)}(X)}(\A(X,\cF_1)_{(k_1,\dots,k_p)},\A(X,\cF_2)_{(k_1,\dots,k_p)})=0.\end{equation}\end{lemma}

\begin{proof}1) We have the following chain of isomorphisms
\begin{multline}\label{adjunctions}\Hom_{\A(X)_{(k_1,\dots,k_p)}}(\A(X,\cF)_{(k_1,\dots,k_p)},\A(X,j_{\eta*}\cG)_{(k_1,\dots,k_p)})\cong\\
\Hom_{\A(X)_{(k_1,\dots,k_p)}}(\A(X,\cF)_{(k_1,\dots,k_p)},g_{\eta*}(\A(X,j_{\eta*}\cO_{\eta})_{(k_1,\dots,k_p)}\stackrel{\bL}{\otimes}_{\cO_{\eta}}\cG_{\eta}))\cong\\
\Hom_{\A(X,j_{\eta*}\cO_{\eta})_{(k_1,\dots,k_p)}}(\bL g_{\eta}^*\A(X,\cF)_{(k_1,\dots,k_p)},\A(X,j_{\eta*}\cO_{\eta})_{(k_1,\dots,k_p)}\stackrel{\bL}{\otimes}_{\cO_{\eta}}\cG_{\eta})\cong\\
\Hom_{\A(X,j_{\eta*}\cO_{\eta})_{(k_1,\dots,k_p)}}(\A(X,j_{\eta*}\cO_{\eta})_{(k_1,\dots,k_p)}\stackrel{\bL}{\otimes}_{\cO_{\eta}}\bL j_{\eta}^*(\cF),\\
\A(X,j_{\eta*}\cO_{\eta})_{(k_1,\dots,k_p)}\stackrel{\bL}{\otimes}_{\cO_{\eta}}\cG_{\eta})=0,\end{multline}
since $\bL j_{\eta}^*(\cF)=0.$

2) This is evident, as in Lemma \ref{description} 2).

3) Using 1) and the chain \eqref{adjunctions}, we see that
\begin{multline}\Hom_{D(\A(X)_{(k_1,\dots,k_p)})/D_{\leq(k_0-1)}(X)}(\A(X,\cF_1)_{(k_1,\dots,k_p)},\A(X,\cF_2)_{(k_1,\dots,k_p)})\cong\\
\Hom_{\A(X)_{(k_1,\dots,k_p)}}(\A(X,\cF_1)_{(k_1,\dots,k_p)},\A(X,j_{\eta_2*}\cF_{2\eta_2})_{(k_1,\dots,k_p)})\cong\\
\Hom_{\A(X,j_{\eta_2*}\cO_{\eta_2})_{(k_1,\dots,k_p)}}(\A(X,j_{\eta_2*}\cO_{\eta_2})_{(k_1,\dots,k_p)}\stackrel{\bL}{\otimes}_{\cO_{\eta_2}}\bL j_{\eta_2}^*(\cF_1),\\
\A(X,j_{\eta_2*}\cO_{\eta_2})_{(k_1,\dots,k_p)}\stackrel{\bL}{\otimes}_{\cO_{\eta_2}}\cF_{2\eta_2})=0,
\end{multline}
since $\bL j_{\eta_2}^*(\cF_1)=0.$
\end{proof}

For convenience put $\cB:=\bR\End_{D(\A(X)_{(k_1,\dots,k_p)})/D_{\leq(k_0-1)}(X)}(\A(X)_{(k_1,\dots,k_p)}).$ Let $\cT\subset D(\cB)$ (resp. $\cT_{\eta}\subset D(\cB)$, $\eta\in S(X)_{(k_0)}$) be the subcategory
classically generated by the image of $D^b_{coh,\leq k_0}(X)$ (resp. $D^b_{coh,\overline{\eta}}$). Then, by Lemma \ref{description2} and
Proposition \ref{product_gen} we have that
\begin{equation}\label{product2}\widehat{\cB}_{\cT}\cong\prod\limits_{\eta\in S(X)_{(k_0)}}\widehat{\cB}_{\cT_{\eta}}.\end{equation}

Fix some $\eta\in S(X)_{(k_0)}.$ We claim that $\widehat{\cB}_{\cT_{\eta}}\cong \lim_n\A(X,j_{\eta*}(\cO_{\eta}/\m_{\eta}^n))_{(k_1,\dots,k_p)}.$
This can be shown as follows. Put $\cB_n:=\A(X,j_{\eta*}(\cO_{\eta}/\m_{\eta}^n))_{(k_1,\dots,k_p)}.$ Then $\{\cB_n\}_{n\in\N}\in\dgalg_k^{\N^{op}},$
 and we have a compatible system of morphisms $\cB\to\cB_n.$ Denote by $g_{\eta,n}:\A(X)_{(k_1,\dots,k_p)}\to\cB_n$ the natural map.
Put $Y:=\overline{\eta}.$
Denote by $Y_l$ the infinitesimal neighborhoods of $Y,$ and $\iota_l:Y_l\to X$ the inclusions.
We have natural isomorphisms
\begin{multline}\cB_n\cong \bR\Hom_{\A(X)_{(k_1,\dots,k_p)}}(\A(X)_{(k_1,\dots,k_p)},g_{\eta,n*}(\A(X,j_{\eta*}(\cO_{\eta}/\m_{\eta}^n))_{(k_1,\dots,k_p)}))\cong\\
\bR\Hom_{D(\A(X)_{(k_1,\dots,k_p)})/D_{\leq(k_0-1)}(X)}(\A(X)_{(k_1,\dots,k_p)},g_{\eta,n*}(\A(X,j_{\eta*}(\cO_{\eta}/\m_{\eta}^n))_{(k_1,\dots,k_p)}))\cong\\
\bR\Hom_{D(\A(X)_{(k_1,\dots,k_p)})/D_{\leq(k_0-1)}(X)}(\A(X)_{(k_1,\dots,k_p)},\A(X,\iota_{n*}\cO_{Y_n})_{(k_1,\dots,k_p)})
\text{ in }D(\cB).\end{multline}

Further, denote by
\begin{equation}\Phi:D(\A(X)_{(k_1,\dots,k_p)})\to D(X)\end{equation}
the functor which is right adjoint to $\A(X,-)_{(k_1,\dots,k_p)}.$

\begin{lemma}\label{phi_preserves}Let $\cF\in D^b_{coh,Y}(X).$ Then
\begin{equation}\Phi(\A(X,j_{\eta*}\cF_{\eta})_{(k_1,\dots,k_p)})\in D_Y(X).\end{equation}\end{lemma}

\begin{proof}We may assume that $\cF=\iota_{1*}\cF'$ for some object $\cF'\in D^b_{coh}(Y).$ Denote by
$\pi:\A(X)_{(k_1,\dots,k_p)}\to \A(Y)_{(k_1,\dots,k_p)}$ the natural projection, and let
\begin{equation}\Psi:D(\A(Y)_{(k_1,\dots,k_p)})\to D(Y)\end{equation}
be right adjoint to $\A(Y,-)_{(k_1,\dots,k_p)}.$ Note the isomorphism of functors
\begin{equation}\bL\pi^*\A(X,-)_{(k_1,\dots,k_p)}\cong\A(Y,\bL\iota_1^*(-))_{(k_1,\dots,k_p)}.\end{equation}
We have the following chain of isomorphisms:
\begin{multline}\Phi(\A(X,j_{\eta*}\cF_{\eta})_{(k_1,\dots,k_p)})\cong \Phi(\pi_*(\A(Y,\cF'\otimes k(\eta))_{(k_1,\dots,k_p)}))\cong\\
 \iota_{1*}\Psi(\A(Y,\cF'\otimes k(\eta))_{(k_1,\dots,k_p)})\in D_Y(X).\end{multline}
 This proves Lemma.\end{proof}

Now, for each object $\cF\in D^b_{coh,Y}(X)$ we have
\begin{multline}\hocolim_n \bR\Hom_{\cB}(\cB_n,\\
\bR\Hom_{D(\A(X)_{(k_1,\dots,k_p)})/D_{\leq(k_0-1)}(X)}(\A(X)_{(k_1,\dots,k_p)},
\A(X,\cF)_{(k_1,\dots,k_p)}))\cong\\
\hocolim_n \bR\Hom_{D(\A(X)_{(k_1,\dots,k_p)})/D_{\leq(k_0-1)}(X)}(\A(X,\iota_{n*}\cO_{Y_n})_{(k_1,\dots,k_p)},
\A(X,\cF)_{(k_1,\dots,k_p)})\cong\\
\hocolim_n \bR\Hom_{\A(X)_{(k_1,\dots,k_p)}}(\A(X,\iota_{n*}\cO_{Y_n})_{(k_1,\dots,k_p)},
\A(X,j_{\eta*}(\cF_{\eta}))_{(k_1,\dots,k_p)})\cong\\
\hocolim_n \bR\Hom_{D(X)}(\iota_{n*}\cO_{Y_n},
\Phi(\A(X,j_{\eta*}(\cF_{\eta}))_{(k_1,\dots,k_p)})).\end{multline}
By Lemma \ref{phi_preserves}, the last object of $D(k)$ is isomorphic to
\begin{multline}\bR\Hom_{D(X)}(\cO_X,
\Phi(\A(X,j_{\eta*}(\cF_{\eta}))_{(k_1,\dots,k_p)}))\cong\\
\bR\Hom_{\A(X)_{(k_1,\dots,k_p)}}(\A(X)_{(k_1,\dots,k_p)},
\A(X,j_{\eta*}(\cF_{\eta}))_{(k_1,\dots,k_p)})\cong\\
\bR\Hom_{D(\A(X)_{(k_1,\dots,k_p)})/D_{\leq(k_0-1)}(X)}(\A(X)_{(k_1,\dots,k_p)},
\A(X,\cF)_{(k_1,\dots,k_p)}).
\end{multline}

Therefore, by Lemma \ref{homot_limits}, we have that
\begin{equation}\widehat{\cB}_{\cT_{\eta}}\cong\holim_n\cB_n\cong\lim_n\A(X,j_{\eta*}(\cO_{\eta}/\m_{\eta}^n))_{(k_1,\dots,k_p)}.\end{equation}
According to \eqref{product2}, we have
\begin{equation}\widehat{\cB}_{\cT}\cong\prod\limits_{\eta\in S(X)_{(k_0)}}\lim_n\A(X,j_{\eta*}(\cO_{\eta}/\m_{\eta}^n))_{(k_1,\dots,k_p)}\cong \A(X)_{(k_0,\dots,k_p)}.\end{equation}
Hence, we have equivalences
\begin{equation}\widehat{(D(\A(X)_{(k_1,\dots,k_p)})/D_{\leq(k_0-1)}(X))}_{D^b_{coh,\leq k_0}}\cong D(\widehat{\cB}_{\cT})\cong D(\A(X)_{(k_0,\dots,k_p)}),\end{equation}
and commutativity of \eqref{commut4} is straightforward to check.
\end{proof}

\end{document}